\documentclass[11pt]{article}
\usepackage{amsmath, amsfonts, amsthm, amssymb, color}

\newtheorem{theorem}{Theorem}

\newtheorem{lemma}[theorem]{Lemma}

\def\ind{{\mathbf{1}}}
\def\EE{{\mathbb E}}
\def\PP{{\mathbb P}}

\def\dist{\text{dist}}

\def\HH{{\mathcal H}}
\def\cG{\mathcal{G}}
\def\RR{{\mathbb R}}

\def\cF{{\mathcal F}}
\def\ind{{\mathbf{1}}}
\def\EE{{\mathbb E}}
\def\PP{{\mathbb P}}

\def\dist{\mathrm{dist}}

\def\cG{\mathcal{G}}

\def\tx{{\rm tx}}

\newcommand{\eps}{\varepsilon}

\newcommand{\Bin}{\mathsf{Bin}}

\newcommand{\Po}{\mathsf{Po}}

\title{Shortest-weight paths in random regular graphs}

\begin{document}
\author{Hamed Amini\footnote{EPFL, Lausanne, Switzerland; hamed.amini@epfl.ch} \and  Yuval Peres\footnote{Microsoft Research, Redmond, Washington, USA; peres@microsoft.com} }
\date{}
\maketitle

\begin{abstract}
Consider a random regular graph with degree $d$ and of size $n$. Assign to each edge an i.i.d.\ exponential random variable with mean one.
In this paper we establish a precise asymptotic expression for the maximum number of edges on the shortest-weight paths between a fixed vertex and all the other vertices, as well as between any pair of vertices.
Namely, for any fixed $d \geq 3$, we show that the longest of these shortest-weight paths has about $\widehat{\alpha}  \log n$ edges where $\widehat{\alpha} $ is the unique solution of the equation $\alpha \log\left(\frac{d-2}{d-1}\alpha\right) - \alpha = \frac{d-3}{d-2}$, for $\alpha > \frac{d-1}{d-2}$.

\end{abstract}

\section{Introduction}

The focus of this paper is on first passage percolation on a random regular graph,
namely on $G\sim\cG(n,d)$, a graph uniformly distributed  over the set of all graphs on $n$ vertices $[n]:= \{1,\dots,n\}$, in which every vertex has degree $d$, for $d \geq 3$ and $n$ large. We assume that each edge in this graph has an i.i.d.\ exponential weight with mean one. We consider the shortest-weight paths between any pair of vertices of this graph, and establish that the longest of
these shortest-weight paths has about $\widehat{\alpha}  \log n$ edges for some positive constant $\widehat{\alpha}$ depending on $d$ that we will shortly define.
We also derive a similar precise asymptotic expression for the maximum number of edges on the shortest-weight paths between a fixed vertex and all the other vertices, see Theorem~\ref{thm-main} for the exact statement.

Let $G=(V,E,w)$ be a weighted graph, defined as the data of a graph $G = (V,E)$ and a collection of weights $w = \{w_e\}_{e\in E}$ associated to each edge $e \in E$. For two vertices $a,b \in V$, the weighted distance between $a$ and $b$ is given by $$\dist_w(a,b) = \min_{\pi\in\Pi(a,b)}\sum_{e\in \pi} w_e\:,$$
where the minimum is taken over the set $\Pi(a,b)$ of all paths between $a$ and $b$ in the graph. For $a,b \in V$ we denote by $\pi(a,b)$ the shortest-weight path between $a$ and $b$.

We define the function $f: \RR^+ \rightarrow \RR$ as follows
\begin{eqnarray}
f(\alpha):= \alpha \log\left(\frac{d-2}{d-1}\alpha\right) - \alpha  + \frac{1}{d-2}.
\end{eqnarray}
Note that
$f'(\alpha) =  \log\left(\frac{d-2}{d-1}\alpha\right)$
is positive for $\alpha > \frac{d-1}{d-2}$, and $f\left(\frac{d-1}{d-2}\right) = -1$.

\noindent We let $\alpha^*$ and $\widehat{\alpha} $ be respectively the unique solutions to $f(\alpha)=0$ and $f(\alpha)=1$ for $\alpha > \frac{d-1}{d-2}$.

\medskip

\noindent The main result of this paper is the following theorem.

\begin{theorem}\label{thm-main}
Fix $d \geq 3$ and let $G\sim\cG(n,d)$ be a weighted random $d$-regular graph with $n$ vertices and i.i.d.\ rate one exponential variables on its edges. Then, as $n \rightarrow \infty$, we have
\begin{eqnarray}
\frac{\max_{j \in [n]} |\pi(1,j)|}{\log n} \stackrel{p}{\longrightarrow} \alpha^*,
\end{eqnarray}
and
\begin{eqnarray}
\frac{\max_{i,j \in [n]} |\pi(i,j)|}{\log n} \stackrel{p}{\longrightarrow} \widehat{\alpha} ,
\end{eqnarray}
where $\stackrel{p}{\longrightarrow}$ denotes the convergence in probability.
\end{theorem}

In order to compare our result with the existing ones, we reproduce here a result of Bhamidi, van der Hofstad and Hooghiemstra \cite{BHH09} concerning the number of edges in the shortest-weight path between two uniformly chosen nodes (as well as the weighted distance); see also \cite{salez12} for the joint distribution of (weighted) distances in random regular graphs. Remark that, the following theorem is stated in \cite{BHH09} in a more general setting (random graphs with i.i.d.\ degrees).

\begin{theorem}[Bhamidi, van der Hofstad and Hooghiemstra~\cite{BHH09}]\label{thm-FPP}
Fix $d \geq 3$ and let $G\sim\cG(n,d)$ be a random $d$-regular graph with $n$ vertices and i.i.d.\ rate one exponential variables on its edges. Then, as $n \rightarrow \infty$,
\begin{eqnarray}
\frac{|\pi(1,2)| - \gamma \log n}{\sqrt{\gamma \log n}} \stackrel{d}{\longrightarrow} Z,
\end{eqnarray}
where $Z$ has a standard normal distribution, $\gamma = \frac{d-1}{d-2}$, and $\stackrel{d}{\longrightarrow}$ denotes the convergence in distribution.
Furthermore, there exists a non-degenerate random variable $W$ such that
\begin{eqnarray}
\dist_w(1,2) - \frac{1}{d-2} \log n \stackrel{d}{\longrightarrow} W.
\end{eqnarray}
\end{theorem}

By the above theorem, the ratio of the length and the weight along a shortest-weight path between two (uniformly chosen) nodes is asymptotically $d-1$ while this ratio for a minimum length path between two nodes is asymptotically $1$. Our proof of Theorem~\ref{thm-main} (see Section~\ref{sec:lower}) implies that, there exists with high probability (that is, with probability tending to $1$ as $n \to \infty$) shortest-weight paths of length about $\widehat{\alpha}  \log n$ whose total weight is about $\frac{1}{d-2} \log n$ (typical weighted distance between two uniformly chosen nodes).  This means that, for these paths,  the ratio of the length and the weight is even larger, i.e., asymptotically $(d-2)\widehat{\alpha}$!

For completeness, we also include results of Ding, Kim, Lubetzky and Peres \cite{DKLP09} concerning the weighted diameter in random regular graphs; see also \cite{amle11} for a generalization.

\begin{theorem} [Ding, Kim, Lubetzky and Peres~\cite{DKLP09}]\label{thm-DKLP}
Fix $d \geq 3$ and let $G\sim\cG(n,d)$ be a random $d$-regular graph with $n$ vertices and i.i.d.\ rate one exponential variables on its edges. Then, as $n \rightarrow \infty$, we have
\begin{eqnarray}
\frac{\max_{j \in [n]} \dist_w(1,j)}{\log n} \stackrel{p}{\longrightarrow} \frac{1}{d-2} + \frac{1}{d},
\end{eqnarray}
and
\begin{eqnarray}
\frac{\max_{i,j \in [n]} \dist_w(i,j)}{\log n} \stackrel{p}{\longrightarrow}\frac{1}{d-2} + \frac{2}{d}.
\end{eqnarray}
\end{theorem}

In particular, the result of \cite{DKLP09} implies that there exists with high probability shortest-weight paths of length about $\frac{d-1}{d-2}  \log n$ (the same as the length between two uniformly chosen nodes, see Theorem~\ref{thm-FPP}) whose total weight is about $\left(\frac{1}{d-2} + \frac{2}{d}\right) \log n$. This result is used in \cite{ADL11} to analyze an asynchronous randomized broadcast algorithm for random regular graphs.

\paragraph{Related work.}

First passage percolation model has been mainly studied  on lattices motivated by its subadditive property and its link to a number of other stochastic processes, see e.g., \cite{grimkest84, kesten86, hagpem98} for a more detailed discussion.
First passage percolation with exponential weights has received substantial attention, in particular on the complete graph \cite{FG85, Jan99, AldBha2007, ABL10, Ding11, Peres10}, and more recently on random graphs \cite{bhamidi08, BHH09, BHH10, DKLP09, amle11, Antun11}. In particular, Janson \cite{Jan99} considered the
case of the complete graph with fairly general i.i.d.\ weights on edges, including the exponential distribution with parameter one. It is shown that, when $n$ goes to infinity, the
asymptotic distance  for two given points is $\log n / n$, that the maximum distance if one point is fixed and the other varies is $2\log n / n$, and the maximum distance over all pairs of points is $3\log n / n$. He also derives asymptotic results for the corresponding number of hops or hopcount (the number of edges on the paths with the smallest weight). It is shown that (when $n$ goes to infinity) the number of hops is $\log n$ for two given nodes, and the maximum hops if one point is fixed and the other varies is $e\log n$. More recently, Addario-Berry, Broutin and Lugosi \cite{ABL10} showed that the longest of these shortest-weight paths in a complete graph has about $\widetilde{\alpha} \log n$ edges where $\widetilde{\alpha} \sim 3.5911$ is the unique solution of the equation $\alpha \log(\alpha) - \alpha =1$, which answered a question posed by Janson \cite{Jan99}. Note that $\alpha^* \to e$ and $\widehat{\alpha} \to \widetilde{\alpha}$ as $d \to \infty$.

\paragraph{Organization of the paper.}

The remainder of the paper is organized as follows.
In the next section we provide several preliminary facts on random regular graphs.
We also consider in this section the exploration process for configuration model which consists in growing balls (neighborhoods) simultaneously from each vertex. In addition, the section provides some necessary notations and definitions that will be used throughout the paper. Sections~\ref{sec:upper} and~\ref{sec:lower} form the heart of the proof. We first prove that the above bound is an upper bound in Sections~\ref{sec:upper}. The final section provides the corresponding lower bound using the second moment method, applied to a suitably defined set of shortest paths with special properties that make them amenable to analysis.

\paragraph{Basic notations.}

Let $\{ X_n \}_{n \in \mathbb{N}}$ be a sequence of real-valued random variables on a sequence of probability spaces
$\{ (\Omega_n, \mathbb{P}_n)\}_{n \in \mathbb{N}}$.
If $c \in \mathbb{R}$ is a constant, we write $X_n \stackrel{p}{\rightarrow} c$ to denote that $X_n$ \emph{converges in probability to $c$}.
That is, for any $\eps >0$, we have $\mathbb{P}_n (|X_n - c|>\eps) \rightarrow 0$ as $n \rightarrow \infty$.
\\
Let $\{ a_n \}_{n \in \mathbb{N}}$ be a sequence of real numbers that tends to infinity as $n \rightarrow \infty$.
We write $X_n = o_p (a_n)$ if $|X_n|/a_n$ \emph{converges to 0 in probability}.
Additionally, we write $X_n = O_p (a_n)$ to denote that for any positive-valued function $\omega (n) \rightarrow \infty$,
as $n \rightarrow \infty$, we have $\mathbb{P} (|X_n|/a_n \geq \omega (n)) = o(1)$.
If $\mathcal{E}_n$ is a measurable subset of $\Omega_n$, for any $n \in \mathbb{N}$, we say that the sequence
$\{ \mathcal{E}_n \}_{n \in \mathbb{N}}$ occurs \emph{with high probability  (w.h.p.)}\ if $\mathbb{P} (\mathcal{E}_n) = 1-o(1)$, as
$n\rightarrow \infty$.

\noindent 
The notation $\Bin (k,p)$ denotes a binomially distributed random variable corresponding to the number of
successes of a sequence of $k$ independent Bernoulli trials each having probability of success equal to $p$.

\noindent We recall here that for two real-valued random variables $A$ and $B$, we say $A$ is stochastically dominated by $B$ and write $A\leq_{st} B$ if
for all $x$, we have $\PP(A\geq x)\leq \PP(B\geq x)$. If $C$ is another random
variable, we write $A\leq_{st} (B\,|\,C)$ if for all $x$, $\PP(A\geq x)\leq
\PP(B\ge x\,|\,C)$ almost surely.

\section{Preliminaries}

\subsection{Configuration model}

We recall first the setup of
the {\it configuration model} (CM), as  introduced by Bender and Canfield\cite{BC78} and Bollob\'as \cite{Bol01}. To
construct a graph using this method, to each of the $n$ (even)
vertices allocate $d$ distinct half-edges, and select a uniform
perfect matching on these points. When a half-edge of $i$ is paired with a half-edge of $j$, we interpret this as an edge between
$i$ and $j$.

The random graph obtained following this procedure may not be simple, i.e., may
contain self-loops due to the pairing of two half-edges of $i$, and multi-edges due to the existence
of more than one pairing between two given nodes. Conditional on
the event that the graph produced is simple, it is uniformly distributed over the set of all $d$-regular graphs on $n$ vertices.
The probability of this event is
uniformly bounded away from zero, equivalent to
$(1+o(1))\exp\left(\frac{1-d^2}{4}\right)$ as $n$ tends to infinity \cite{Wor99}. Hence,
any event that holds w.h.p.\ for the graph obtained
via the configuration model also holds w.h.p.\ for $G\sim\cG(n,d)$.

Note that the assumption $d \geq 3$ implies that $G\sim\cG(n,d)$ is connected with high probability~\cite{Bol01, Wor99}. We will assume this in what follows.

The advantage of using the configuration model is that it allows one to construct the graph gradually, exposing the edges of the perfect matching one at a time. This way, each additional edge is uniformly distributed among all possible edges on the remaining (unmatched) half-edges.

\subsection{Neighborhoods and tree excess}

For $a,b \in V$, let $\dist(a,b) = \dist_G(a,b)$ denote the typical distance between $a$ and $b$. For a vertex $a \in V$ and an integer number $m$, the $m$-step neighborhood of $a$, denoted by $B(a,m)$ and its boundary $\partial B(a,m)$, are defined as
\begin{eqnarray}
B(a,m) := \{v \in V \mid \dist(a,v) \leq m\}, \ \mbox{and}, \ \partial B(a,m) := B(a,m)\backslash B(a,m-1).
\end{eqnarray}

For a vertex $a\in V$ and a real number $t>0$, the $t$-radius neighborhood of $a$ in the weighted graph, or the ball of radius $t$ centered at $a$,  is defined as
$$B_w(a,t) := \bigl\{\,b,\:\dist_w(a,b)\leq t\,\bigr\}.$$

The first time $t$ where the ball $B_w(a,t)$ reaches size $k+1\geq 1$ will be denoted by $T_k(a)$, i.e.,
$$T_k(a) = \min\, \bigl\{\,t: |B_w(a,t)| \geq k+1 \,\bigr\}, \qquad T_0(a) = 0.$$

Note that there is a vertex in $B_w(a,T_k(a))$ which is not in any ball of smaller radius around $a$. When the weights are i.i.d.\ according to a random variable with continuous density, this vertex is in addition unique with probability one. We will assume this in what follows.
Let $v_k(a)$ denote this node. Furthermore, let $H_k(a)$ denote the number of edges (hopcounts) in the shortest path between the node $a$ and $v_k(a)$, i.e., the generation of $v_k(a)$.

For a connected graph $F$, the tree excess of
$F$ is denoted by $tx(F)$, which is the maximum number of edges that can be deleted from $F$ while still keeping it connected. By an abuse of notation, for a subset $W \subseteq V$, we denote by $tx(W)$ the tree excess of the induced subgraph $G[W]$ of $G$ on $W$. (If $G[W]$ is not connected, then $tx(W) :=\infty$.)

We need the following lemma which demonstrates the well known locally tree-like properties of $G\sim G(n,d)$ for $d \geq 3$.

\begin{lemma}\label{lem-treeexcess}
Let $G\sim G(n,d)$ for some fixed $d \geq 3$, and let $m=\lfloor \frac{1}{5} \log_{d-1} n\rfloor$. Then w.h.p., $\tx (B(u, m)) \leq 1$ for all $u \in V(G)$.
\end{lemma}
\begin{proof}
See \cite[Lemma 2.1]{LubetzkySly10} .
\end{proof}

Consider now the growing balls $B_w(a,T_k(a))$ for $0\leq k\leq n-1$ centered at $a$ and let $X_k(a)$ be the tree excess of $B_w(a,T_k(a))$, i.e.,
$$X_k(a) : = tx\,(\,B_w(a,T_k(a))\,).$$

The number of edges crossing the boundary of the ball $B_w(a, T_k(a))$ is denoted by $S_k(a)$. A simple calculation shows that (for $G\sim\cG(n,d)$)
\begin{eqnarray}\label{eq:boundary}
\label{eq:S(a)}S_k(a) = d + (d-2)k  - 2 X_k(a).
\end{eqnarray}

\subsection{Shortest-weight paths on a tree}
Assume we have positive integers $d_1, d_2, ...$. We consider the following construction of a branching process (with these degrees) in discrete time:
\begin{itemize}
  \item At time 0, start with one alive vertex (the root);
  \item At each time step $k$, pick one of the alive vertices at random, this vertex dies giving birth to $d_k$ children.
\end{itemize}

This type of random tree is known as (random) increasing trees which have been well-studied, see e.g. \cite{Broutin08, Drmota09, Flajolet09}.
We will need the following basic result, the proof of which is easy and can be found for example in \cite[Proposition 4.2]{BHH09}.
Let $s_k:= d_1 + ... + d_k - (k-1)$.

\begin{lemma}\label{lem-gen-SWG}
Pick an alive vertex at time $k \geq 1$¸ uniformly
at random among all vertices alive at this time. Then, the generation of the $k$-th chosen vertex is equal in distribution to
$$G_k \stackrel{d}{=} \sum_{i=1}^{k} I_i,$$
where $\{I_i\}_{i=1}^{\infty}$ are independent Bernoulli random variables with parameter
$$\PP(I_i = 1) = \frac{d_i}{s_i}.$$
\end{lemma}

In what follows, instead of taking a graph at random and then analyzing the balls, we use a standard coupling argument in random graph theory which allows to build the balls and the graph at the same time.
Fix two vertices, say $u$ and $v$. We grow the balls around these vertices simultaneously at rate 1, so that at time $t$, $B_w(u, t)$ and $B_w(v, t)$ are the constructed balls from $u$ and $v$. When these two balls intersect via the formation of an edge $(u^*_v,v^*_u)$ between two vertices $u^*_v \in B_w(u, .)$ and $v^*_u \in B_w(v, .)$, then the shortest-weight path between the two vertices has been found. Furthermore, we have
$$|\pi(u,v)| = |\pi(u,u^*_v)| + |\pi(v,v^*_u)| + 1.$$

\subsection{The exploration process}\label{sse:explor}

Fix a vertex $a$, and consider the following continuous-time exploration process. At time $t=0$, we have a neighborhood consisting only of $a$, and
for $t>0$, the neighborhood is precisely $B_w(a, t)$. We now give an equivalent description of this process.

\begin{itemize}
\item {Start with $B = \{a\}$, where $a$ has $d$ half-edges. For each half edge, decide (at random depending on the previous choices) if the half-edge is matched to a half-edge adjacent to $a$ or not. Reveal the matchings consisting of those half-edges adjacent to $a$ which are connected amongst themselves (creating self-loops at $a$) and assign weights independently at random to these edges. The remaining  unmatched half-edges adjacent to $a$ are stored in a list $L$. (See the next step including a more precise description of this first step.)}

\item Repeat the following exploration step as long as the list $L$ is
  not empty.
\item[] Given there are $\ell\geq 1$ half-edges in the current list, say
  $L=(h_1,\dots, h_\ell)$, let $\Psi \sim \mathrm{Exp}(\ell)$ be an
  exponential variable with mean $\ell^{-1}$. After time $\Psi$ select
  a half-edge from $L$ uniformly at random, say $h_i$. Remove $h_i$
  from $L$ and match it to a
  uniformly chosen half-edge in the entire graph excluding $L$, say $h$. Add
  the new vertex (connected to $h$) to $B$ and reveal the
  matchings (and weights) of any of its half-edges whose matched half-edge is
  also in $B$. More precisely, let $2x$ be the number of already matched half-edges in $B$ (including
  the matched half-edges $h_i$ and $h$). There is a
  total of $dn-2x$ unmatched half-edges. Consider one of the
  $d-1$ half-edges of the new vertex (excluding $h$ which is connected
  to $h_i$); with probability $(\ell-1)/(dn-2x-1)$ it is matched with a
  half-edge in $L$ and with the complementary
  probability it is matched with an unmatched half-edge outside
  $L$. In the first case, match it to a uniformly chosen half-edge of
  $L$ and remove the corresponding half-edge from $L$. In the second
  case, add it to $L$. We proceed in the similar manner for all the
  $d-1$ half-edges of the new vertex.
\end{itemize}

To verify the validity of the above process, let $B_t(a)$ and $L(a,t)$ be respectively the set of vertices and the
list generated by the above procedure at time $t$, where $a$ is the
initial vertex. Considering the usual configuration model and using the memoryless
property of the exponential distribution, we have $B_w(a,t)=B_t(a)$ for all $t$. To see this, we can continuously grow
the weights of the half-edges $h_1, \dots, h_{\ell}$ in $L$ until one
of their rate $1$ exponential clocks fire. Since the minimum of $\ell$
i.i.d exponential variables with rate 1 is exponential with rate
$\ell$, this is the same as choosing uniformly a half-edge $h_i$ after
time $\Psi$ (recall that by our conditioning, these $\ell$ half-edges
do not pair within themselves). Note that the final weight of an edge
is accumulated between the time of arrival of its first half-edge and
the time of its pairing (except edges going back into $B$ whose
weights are revealed immediately). Then the equivalence follows from
the memoryless property of the exponential distribution.

Note that $T_i(a)$ is the time of the $i$-th  exploration step in the
above continuous-time exploration process.
Assuming $L(a,T_i(a))$ is not empty, at time $T_{i+1}(a)$, we match a uniformly
chosen half-edge from the set $L(a,T_i(a))$ to a uniformly chosen
half-edge among all other half-edges, excluding those in
$L(a,T_i(a))$. Let $\cF_{t}$ be the $\sigma$-field generated by the
above process until time $t$.
Given $\cF_{T_i(a)}$, $T_{i+1}(a)-T_i(a)$ is an
exponential random variable with rate $S_i(a)$ given by Equation (\ref{eq:S(a)}) which is equal to $|L(a,T_i(a))|$ the size of the list consisting of unmatched half-edges in $B_{T_i(a)}(a)$. In other words,
\[ \bigl( T_{i+1}(a)-T_i(a)\, |\, \cF_{T_i(a)} \bigr) \stackrel{d}{=} \mathrm{Exp}(S_i(a)), \]
this is true since the minimum of $k$ i.i.d.\ rate one exponential  random variables is an exponential of rate k.

We will need the following coupling lemma the proof of which can be found in \cite[Proposition 4.5]{BHH09}.
\begin{lemma}[Coupling shortest-weight graphs on a tree and CM]
For a uniformly chosen vertex $u$, we have (for all $k \geq 1$)
$$H_k(u) \stackrel{d}{=} \sum_{i=1}^{k} I_i,$$
where $\{I_i\}_{i=1}^{\infty}$ are independent Bernoulli random variables with parameter
$$\PP(I_i = 1) = \frac{d-1}{S_i(u)},$$
and $S_i(u)$ is given by Equation~\ref{eq:boundary}.
\end{lemma}

\section{Proof of the upper bound}\label{sec:upper}
In this section we present the proof of the upper bound for Theorem \ref{thm-main}.

As described above, we grow the balls around each vertex simultaneously (at rate one) so that at time $t$, $B_t(a) = B_w(a,t)$ is the ball constructed from vertex $a$.


We let $q:= \lfloor 2 \sqrt{dn \log n} \rfloor$. The following lemma says that for all vertices $u$ and $v$, the growing balls centered at $u$ and $v$ intersect w.h.p.\ provided that they contain each at least $q$ nodes. More precisely,

\begin{lemma}\label{lemup-intersect}
We have with high probability
\begin{eqnarray}
B_w(u, T_q(u)) \cap B_w(v, T_q(v)) \neq \emptyset, \text{ for all } u \text{ and  } v.
\end{eqnarray}
\end{lemma}
\noindent For the sake of readability, we postpone the proof of the lemma to the end of this section.

Fix two vertices $u$ and $v$. Let
$$C(u,v):= \min\{k \geq 0: B_w(u, T_k(u)) \cap B_w(v, T_k(v)) \neq \emptyset\},$$
be the first time that $B_w(u, T_.(u))$ and $B_w(v, T_.(v))$ share a vertex. Thus, by the above lemma w.h.p.\ $C(u,v) < q$ for all $u$ and $v$. Let us denote by $Q$ the following event:
$$Q := \{C(u,v)<q \mbox{ for all } u \mbox{ and } v\}.$$

Consider now the exploration process started at a vertex $u$. We will need to find lower bounds for $S_k(u)$ in the range $1 \leq k \leq q$. We let $r:= \lfloor(\log n)^3\rfloor$.

By the uniform choice of the matching, for every $k \geq 0$, the number of half-edges introduced by the new vertex at time $T_{k+1}(u)$ and connecting back to $B_w(u, T_k(u))$ (given $\cF_{T_k(u)}$) is stochastically dominated by a binomial variable
$$\Bin(d-1, \alpha), \text{ where } \alpha = \frac{d+(d-2)(k+1)}{dn-2k} \leq \frac{k+2}{n},$$
where the above inequality is valid for $k \leq \frac{n}{2}-5$.
Therefore, the tree excess of $B_w(u, T_k(u))$ is stochastically dominated by a binomial variable $\Bin(dk, \frac{k+2}{n})$.

\noindent We have (for large $n$)
\begin{eqnarray}
\PP(X_r(u) \geq 2)  \leq \PP\left(\Bin \left(dr,\frac{r+2}{n}\right) \ge 2\right) \leq O(\frac{r^4}{n^2}) = o(n^{-3/2}).
\end{eqnarray}

\noindent Moreover, for any $k$ satisfying $r \leq k \leq 2q$, we have by Chernoff's inequality
\begin{eqnarray}
\PP(\{X_k(u) \geq k/\sqrt{r}\}) \leq  \PP\left(\Bin \left(dk,\frac{k+2}{n}\right) \geq k/\sqrt{r}\right) \leq \exp\left(-\frac{1}{3}k/\sqrt{r}\right) < n^{-5},
\end{eqnarray}
for any sufficiently large $n$, since $k^2/n = o(k/\sqrt{r})$.

\noindent We conclude by a union bound over all $r \leq k \leq 2q$,
$$\PP(\{X_k(u) < k/\sqrt{r}, \ \mbox{for all} \ \ r \leq k \leq 2q\}) \geq 1 - o(n^{-4}).$$

\noindent Define the event
\begin{equation}
R_u := \{X_r(u) \leq 1, \ \mbox{and} \ \ X_k(u) < k/\sqrt{r}, \ \mbox{for all} \ \ r < k \leq 2q\},
\end{equation}
such that $\PP(R_u) \geq 1 - o(n^{-3/2})$ by above inequalities.

\noindent Thus defining $R := \bigcap_{u \in [n]}R_u $, we get by union bound
$$\PP(R)\geq 1- o(n^{-1/2}).$$

\noindent Consider now two uniformly chosen vertices $u$ and $v$.
We have
$$\left(|\pi(u,v)| \mid Q \right)\leq_{st} \left(H_q(u) + H_q(v) \mid Q\right).$$

\noindent Furthermore, we have
$$\left(H_q(u) \mid R, Q \right)\leq_{st} \HH := \sum_{i=1}^{q} I_i,$$
where $\{I_i\}_{i=1}^{\infty}$ are independent Bernoulli random variables with parameter
$$\PP(I_i = 1) = \frac{d-1}{1+(d-2)i},$$
for all $1 \leq i \leq r$, and
$$\PP(I_i = 1) = \frac{d-1}{1+(d-2)i - 2i/\sqrt{r}},$$
for all $r < i \leq q$.

\noindent We conclude
\begin{eqnarray}\label{eq-up-stochdom}
\left(|\pi(u,v)| \mid R, Q \right) \leq_{st} \HH_1 + \HH_2 ,
\end{eqnarray}
where $\HH_1$ and $\HH_2$ are two independent copies of $\HH$ defined above.

\noindent We have the following lemma.
\begin{lemma}\label{lemup-main}
We have for some constant $C$ (depending only on $d$)
$$\PP( \HH_1 + \HH_2> \alpha^{*} (\log n + \log\log n) ) \leq C (n \log n) ^{-1}, $$
and,
$$\PP( \HH_1 + \HH_2> \widehat{\alpha}  (\log n + \log\log n) ) \leq C (n \log n) ^{-2}.$$
\end{lemma}
\noindent We postpone the proof of this lemma to the end of this section.

\medskip

\noindent We conclude by (\ref{eq-up-stochdom}), Lemma~\ref{lemup-main} and  union bound that
$$\PP\left(\max_{j \in [n]} |\pi(1,j)| > \alpha^{*} (\log n + \log\log n) \right) \leq \PP( R ^c) + \PP( Q^c ) + C/\log n,$$
and,
$$\PP\left(\max_{i,j \in [n]} |\pi(i,j)| > \widehat{\alpha}  (\log n + \log\log n) \right) \leq \PP ( R ^c) + \PP( Q ^c) + C/\log^2 n.$$

\noindent Since $R$ and $Q$ hold with high probability, we get (w.h.p.)
\begin{eqnarray*}
\max_{j \in [n]} |\pi(1,j)| \leq \alpha^{*} (\log n + \log\log n),
\end{eqnarray*}
and,
\begin{eqnarray*}
\max_{i,j \in [n]} |\pi(i,j)| \leq \widehat{\alpha}  (\log n + \log\log n).
\end{eqnarray*}

\noindent This completes the proof of the upper bound for Theorem~\ref{thm-main}. 

\noindent We end this section by presenting the proof of Lemma~\ref{lemup-intersect} and Lemma~\ref{lemup-main}.

\begin{proof}[Proof of Lemma~\ref{lemup-intersect}]
Fix two vertices $u$ and $v$. First consider the exploration process for $B_w(u, t)$ until reaching $t = T_q(u)$. We know that w.h.p.\ the event $R$ holds. Conditioned on $R$ we have $$S_q(u) \geq (d-1-o(1)) q.$$

Next, consider the exploration process started at $v$. Each matching adds a uniform half-edge to the neighborhood of $v$. Therefore, the probability that $B_w(v, T_q(v))$ does not intersect
$B_w(u, T_q(u))$ is at most
$$\left(1 - \frac{(d-1-o(1)) q}{dn}\right)^{q} \leq \exp\left( \-4(d-1-o(1)) \log n \right) < n^{-7},$$
for any large $n$. A union bound over $u$ and $v$ completes the proof.

\end{proof}

\begin{proof}[Proof of Lemma~\ref{lemup-main}]
We have for $\lambda>0$,

\begin{eqnarray*}
\EE e^{\lambda \HH_1} = \prod_{i=1}^{r} \left(1 + (e^\lambda-1)\frac{d-1}{1+(d-2)i}\right) \prod_{i=r+1}^{q} \left(1 + (e^\lambda-1)\frac{d-1}{1+(d-2)i - 2i/\sqrt{r} }\right).
\end{eqnarray*}

\noindent Then using the fact that $\log(1+x) \leq x$, we obtain

\begin{eqnarray*}
\frac{1}{2}\log \EE e^{\lambda (\HH_1 + \HH_2)} &=& \sum_{i=1}^{r} \log \left(1 + (e^\lambda-1)\frac{d-1}{1+(d-2)i}\right) \\ && + \sum_{i=r+1}^{q} \log\left(1 + (e^\lambda-1)\frac{d-1}{1+(d-2)i - 2i/\sqrt{r} }\right)\\
&\leq& (e^{\lambda}-1) \left( \sum_{i=1}^{r} \frac{d-1}{1+(d-2)i} + \sum_{i=r+1}^{q} \frac{d-1}{1+(d-2)i - 2i/\sqrt{r} }\right)\\
&\leq& (e^{\lambda}-1)\left( \frac{d-1}{d-2}\sum_{i=1}^{r} \frac{1}{i} + \frac{d-1}{d-2} \frac{1}{1-2r^{-1/2}}\sum_{i=r+1}^{q} \frac{1}{i}\right)\\
&\leq& (e^{\lambda}-1)\frac{d-1}{d-2} \left( 1 + O(r^{-1/2})\right)(\log q + 2)\\
&\leq& (e^{\lambda}-1)\frac{d-1}{d-2} (\log q + 3) .
\end{eqnarray*}

\noindent Recall that $q = \lfloor2 \sqrt{dn \log n}\rfloor$.
Choosing  $\lambda:= \log \left(\frac{d-2}{d-1}\alpha^{*}\right)$, we get

\begin{eqnarray*}
\log \EE e^{\lambda (\HH_1 + \HH_2)} &\leq& \left( \alpha^{*} - \frac{d-1}{d-2} \right) (\log n + \log \log n + \log d + 10).
\end{eqnarray*}

\noindent By Markov's inequality we have
\begin{eqnarray*}
\PP( \HH_1 + \HH_2 > \alpha^{*} (\log n + \log\log n) ) &\leq& \EE e^{\lambda (\HH_1 + \HH_2)} \exp(-\lambda \alpha^{*} (\log n + \log\log n) ) \\
&\leq& \exp \left( ( \alpha^{*} - \frac{d-1}{d-2} ) (\log n + \log \log n + \log d + 10)\right) \\ && \ \ \ \ \  \exp \left( - \alpha^{*} \log (\frac{d-2}{d-1}\alpha^{*}) (\log n + \log \log n)\right) \\ &=& C \exp \left( ( \alpha^{*} - \frac{d-1}{d-2} - \alpha^{*} \log \left(\frac{d-2}{d-1}\alpha^{*}\right))\right) \\ && \ \ \ \ \ \exp\left(\log n + \log \log n \right) \\
&=& C \exp \left( -\log n - \log \log n\right) \\ &=& C (n \log n)^{-1}.
\end{eqnarray*}

\noindent Similarly, by taking $\lambda:= \log \left(\frac{d-2}{d-1}\widehat{\alpha} \right)$ we get

\begin{eqnarray*}
\log \EE e^{\lambda (\HH_1 + \HH_2)} &\leq& \left( \widehat{\alpha}  - \frac{d-1}{d-2} \right) (\log n + \log \log n + \log d + 10),
\end{eqnarray*}
and by Markov's inequality we have
\begin{eqnarray*}
\PP( \HH_1 + \HH_2 > \widehat{\alpha}  (\log n + \log\log n) ) &\leq& \EE e^{\lambda (\HH_1 + \HH_2)} \exp(-\lambda \widehat{\alpha}  (\log n + \log\log n) ) \\
&\leq& \exp \left( ( \widehat{\alpha}  - \frac{d-1}{d-2} ) (\log n + \log \log n + \log d + 10)\right) \\ && \ \ \ \ \  \exp \left( - \widehat{\alpha}  \log (\frac{d-2}{d-1}\widehat{\alpha} ) (\log n + \log \log n)\right) \\ &=& C \exp \left( ( \widehat{\alpha}  - \frac{d-1}{d-2} - \widehat{\alpha}  \log \left(\frac{d-2}{d-1}\widehat{\alpha} \right)\right) \\ && \ \ \ \ \  \exp\left(\log n + \log \log n) \right) \\
&=& C \exp \left( -2 (\log n + \log \log n)\right) \\ &=& C (n \log n)^{-2},
\end{eqnarray*}
as required.

\end{proof}

\section{Proof of the lower bound}\label{sec:lower}
In this section we present the proof of the lower bound for Theorem~\ref{thm-main}.

For $\epsilon >0$ (small enough) we define the function $f_{\epsilon}: \RR^+ \rightarrow \RR$ as follows
\begin{eqnarray}
f_{\epsilon}(\alpha) &:=&  \alpha \log\left(\frac{d-2}{(d-1)(1-\epsilon)}\alpha\right) - \alpha(1-\epsilon)  + \frac{1}{d-2}  \\
&=& f(\alpha) + \alpha (\epsilon - \log (1-\epsilon)).
\end{eqnarray}
Let $\alpha^*_{\epsilon}$ and $\widehat{\alpha} _{\epsilon}$ be respectively the unique solutions to $f_{\epsilon}(\alpha)=0$ and $f_{\epsilon}(\alpha)=1$ for $\alpha > \frac{d-1}{d-2}$.
 Note that $\alpha^*_{\epsilon} < \alpha^*, \widehat{\alpha} _{\epsilon} < \widehat{\alpha} $, and furthermore, $\alpha^*_{\epsilon} \to \alpha^*$ and $\widehat{\alpha} _{\epsilon} \to \widehat{\alpha} $ as $\epsilon \to 0$.

 To prove the lower bound, it suffices to show that for all $\epsilon > 0$, there
exist w.h.p.\ a vertex $a$ such that
\begin{align*}
 |\pi(1,a)| \geq \alpha_{\epsilon}^{*} \log n,
\end{align*}
and there exists w.h.p.\ two vertices $u$ and $v$ such that
\begin{align*}
 |\pi(u,v)| \geq \widehat{\alpha}_{\epsilon} \log n.
\end{align*}

\noindent For a path $\gamma_l = v_0 , e_1 , v_1 , \dots , e_l , v_l$ where $v_{i-1}$ and $v_i$ are endpoints of $e_i$ for all $i \in [l]$, let
$$w(\gamma_l) = \sum_{i=1}^l w(e_i).$$

\noindent We first show that given that a path $P(u,v)$ between $u$ and $v$ has small weight, it is very likely to be the
shortest-weight path between its endpoints. More precisely, we have
the following.

\begin{lemma}\label{lem-lower}
For all $n$ sufficiently large, and any path $\gamma_k =v_0 , e_1 , v_1 , \dots , e_k , v_k$ with $k=O(\log n)$, we have (for all $\epsilon >0$)
$$\PP(\gamma_k \neq \pi(v_0,v_k) \mid w(\gamma_k)\leq \frac{1-\epsilon}{d-2}\log n ) =o(1).$$
\end{lemma}

\noindent For the sake of readability, we postpone the proof of the lemma to the end of this section.  Consider a path $\gamma_l = v_0 , e_1 , v_1 , \dots , e_l , v_l$.
It is easily seen that for $t > 0$, letting $\Po(t)$ denote a Poisson mean $t$ random variable, we have
\begin{eqnarray}
 \PP(w(\gamma_\ell) \leq t) = \PP(\Po(t) \geq \ell) \geq \exp(-t)\frac{t^\ell}{\ell!} = \exp\left(-t + \ell \log t - \log \ell! \right).
\end{eqnarray}

\noindent In the following, we let $\ell=\ell_{\epsilon}$ be large enough such that (by Stirling formula)
$$\log \ell! \leq \ell \log \ell - \ell (1-\epsilon).$$

\noindent Thus we have for $\alpha > 0$,

\begin{eqnarray}\label{eq:feps}
 \nonumber \PP\left(w(\gamma_{\ell}) \leq \frac{\ell (1-\epsilon)}{(d-2)\alpha} \right) &\geq& \exp \left(- \frac{\ell(1-\epsilon)}{(d-2)\alpha} + \ell \log\left( \frac{\ell (1-\epsilon)}{(d-2)\alpha} \right) -  \ell \log \ell + \ell (1-\epsilon)\right)\\
\nonumber &=& \exp \left(- \frac{\ell}{\alpha} \left(  \frac{1-\epsilon}{d-2} + \alpha \log\left( \frac{(d-2)\alpha}{1-\epsilon}\right) -  \alpha (1-\epsilon)\right)\right) \\ &=& \exp \left(- \frac{\ell}{\alpha} \left( -\frac{\epsilon}{d-2} + f_{\epsilon}(\alpha) + \alpha \log(d-1)  \right)\right).
 \end{eqnarray}

\noindent We get for $\alpha= \alpha_{\epsilon}^{*} $ in (\ref{eq:feps}) (since $f_{\epsilon}(\alpha_{\epsilon}^{*}) =0$)
\begin{eqnarray}\label{ineq-proba-one}
\PP\left(w(\gamma_{\ell}) \leq \frac{\ell(1-\epsilon)}{(d-2)\alpha_{\epsilon}^{*}} \right) \geq 
(d-1)^{-\ell} \exp\left(\frac{\epsilon \ell}{(d-2)\alpha_{\epsilon}^{*}}\right) ,
\end{eqnarray}
and  for $\alpha= \widehat{\alpha}_{\epsilon}$ in (\ref{eq:feps}) (since $f_{\epsilon}(\widehat{\alpha}_{\epsilon}) =1$)
\begin{eqnarray}\label{ineq-proba-two}
\PP\left(w(\gamma_{\ell}) \leq \frac{\ell(1-\epsilon)}{(d-2)\widehat{\alpha}_{\epsilon}} \right) \geq 
(d-1)^{-\ell} \exp\left(\left(-1+\frac{\epsilon}{d-2}\right) \frac{\ell}{\widehat{\alpha}_{\epsilon}}\right) .
\end{eqnarray}

\begin{lemma}
Assume $t_n=\lfloor c\log n\rfloor$ for some positive constant $c$. For any function $\omega(n)$ tending to $\infty$ with $n$,  w.h.p., there exists $v \in \partial B(1, t_n )$ such that
$$w(\gamma _1(v)) \leq \frac{t_n(1-\epsilon)}{(d-2)\alpha_{\epsilon}^{*}} + \omega(n),$$ where $\gamma_1 (v)$ denote the path from $1$ to $v$ in $B(1,t_n)$.
\end{lemma}
\begin{proof}
We first prove the lemma for the case $c < \frac{1}{5 \log (d-1)}$.

Consider now $B(1, \lfloor c\log n\rfloor)$ for $c < \frac{1}{5 \log (d-1)} $. By  Lemma~\ref{lem-treeexcess} w.h.p.\ $\tx(B(1, \lfloor c\log n\rfloor) \leq 1$, and then by removing at most one of the children of $1$ (and its descendants) we have the tree structure and then, $|\partial B (1, \lfloor c\log n\rfloor)| \geq (d-1)^{\lfloor c\log n\rfloor}$. In the following we assume that one of the children of node 1 is removed (even if $\tx(B(1, \lfloor c\log n\rfloor) = 0$) such that $|\partial B (1, \lfloor c\log n\rfloor)| = (d-1)^{\lfloor c\log n\rfloor}$.

\noindent Let $t_0 = \log_{d-1}\log \omega(n)$.  Note that for any path $\gamma_{t_0}$ of length $t_0$, by Markov inequality
$$\PP(w(\gamma_{t_0}) \geq \epsilon \omega(n))) \leq \frac{t_0}{\epsilon \omega(n)}.$$
Thus, by union bound, the probability that this would be true for one of the nodes at level $t_0$ of node 1 (i.e., in $\partial B(1, t_0)$) is smaller than
$$(d-1)^{t_0} \frac{t_o}{\epsilon \omega(n)} = \frac{\log \omega(n) \log_{d-1} \log \omega(n)}{\epsilon \omega(n)}, $$which goes to zero as $n$ goes to $\infty$.
Then w.h.p.\ the path from the root ($1$) to all nodes at level $t_0$ has weight smaller that $\epsilon \omega(n)$.

\noindent We assume $\ell = \ell_{\epsilon}$ is large enough such that $\exp\left(\frac{\epsilon \ell}{(d-2)\alpha_{\epsilon}^{*}}\right) >1$. Now consider the following branching process starting from a node $r$ at level $t_0$, i.e., $r \in \partial B(1, t_0)$.

We call a vertex $v$ {\it good}  if either $v$ is the root ($v=r$), or if $v$ lies $\ell$ levels below a good vertex $u$ and $w(\gamma _1(u,v)) \leq \frac{\ell(1-\epsilon)}{(d-2)\alpha_{\epsilon}^{*}}$, where $\gamma_1 (u,v)$ denote the path from $u$ to $v$ (in $B(1,\lfloor c\log n\rfloor)$).

\medskip

\noindent The collection of good nodes form a Galton-Watson tree. Let $Z$ denote the progeny distribution of this process. Without need to calculate its distribution, from (\ref{ineq-proba-one}) we know that
$$\EE Z = (d-1)^{\ell} \PP\left(w(\gamma_{\ell}) \leq \frac{\ell(1-\epsilon)}{(d-2)\alpha_{\epsilon}^{*}} \right) \geq 
\exp\left(\frac{\epsilon \ell}{(d-2)\alpha_{\epsilon}^{*}}\right) >1.
$$
Hence, with some positive probability $q_{\epsilon}$ this process survives. We conclude with probability at least $q_{\epsilon}$ we have a good node at level $\lfloor c\log n\rfloor$ from the root $r$ at level $t_0$. Considering the same process for all nodes at level $t_0$, we conclude that there exists a good vertex at level $\lfloor c\log n\rfloor$, with probability at least (by independence of these processes)
$$1 - (1-q_{\epsilon})^{(d-1)^{t_0}} = 1 - (1-q_{\epsilon})^{\log \omega(n)} \to 1, $$ as $n \to \infty$. Then w.h.p.\ we have a node $v$ at level $t_n$, such that
$$w(\gamma _1(v)) \leq \frac{t_n(1-\epsilon)}{(d-2)\alpha_{\epsilon}^{*}} + \epsilon \omega(n).$$

\noindent This completes the proof of lemma for the case $c < \lfloor \frac{1}{5 \log (d-1)} \rfloor$.

Now consider the case  $c \geq \frac{1}{5 \log (d-1)}$, and let $K$ be an integer such that $c' := c/K < \frac{1}{5 \log (d-1)}$. By previous argument, we know that w.h.p.\ there exists a node $v_1$ at level $\lfloor c' \log n \rfloor$ such that  $w(\gamma _1(v_1)) \leq \frac{\lfloor c' \log n \rfloor (1-\epsilon)}{(d-2)\alpha_{\epsilon}^{*}} + \epsilon \omega(n)$. We know repeat the same argument to find a node $v_2$ at level $\lfloor c' \log n \rfloor$ below of node $v_1$ such that $w(\gamma _1(v_1,v_2)) \leq \frac{\lfloor c' \log n \rfloor(1-\epsilon)}{(d-2)\alpha_{\epsilon}^{*}} + \epsilon \omega(n)$, where $\gamma_1 (u,v)$ denote the path from $u$ to $v$ on $B(1, t)$. Note that the tree excess is again at most one, and the number of nodes at level $t_0$ of node $v_1$ is at least $(d-2)(d-1)^{t_0-1}$ which goes to infinity as $n \to \infty$, and we have the similar arguments.  Now repeating this process $K-1$ times completes the proof.
\end{proof}

\noindent Thus, by above lemma, there exists w.h.p.\ a node $a$ at level $\alpha^*_{\epsilon} \log n$ such that $w(\gamma_1(a)) \leq \frac{1-\epsilon}{d-2} \log n$. By Lemma~\ref{lem-lower}, this path is optimal. We conclude w.h.p.\ there exists a node $a$ such that $\pi(1,a) \geq \alpha^*_{\epsilon} \log n$.

\medskip

We now prove that there exists w.h.p.\ two vertices $u$ and $v$ such that
\begin{align*}
 |\pi(u,v)| \geq \widehat{\alpha}_{\epsilon} \log n.
\end{align*}

\noindent Indeed, we prove that there exists w.h.p.\ a path $\gamma$ of length $\widehat{\alpha}_{\epsilon} \log n$ such that $w(\gamma) \leq \frac{1-\epsilon}{d-2} \log n$. Then again using Lemma ~\ref{lem-lower}, we conclude the proof.

Consider the following exploration process starting from a node $a$. We call a vertex $v$, $a$-good  if either $v$ is the root ($v=a$), or if $v$ lies $\ell$ levels below a good vertex $u$ and $w(\gamma _a(u,v)) \leq \frac{\ell (1-\epsilon)}{(d-2)\widehat{\alpha}_{\epsilon}}$, where $\gamma_a (u,v)$ denote the path from $u$ to $v$ in $B(a, .)$.

\noindent To find the nodes which are $a$-good, we first explore the nodes in $B(a, \ell)$, and we find the set of nodes at this level which are $a$-good. Then, for each of these ($a$-good) nodes, we explore again $\ell$ level behind and we continue the exploration until finding all of the $a$-good nodes. Let us denote by $G_{\ell}(a)$ the explored graph (starting from $a$) to find the set of all $a$-good nodes.

\noindent The following lemma bounds from above the size of $G_{\ell}(a)$.

\begin{lemma} \label{lem-boundLog}
Let $G\sim G(n,d)$ for some fixed $d \geq 3$. Then there exists a constant $A$ such that w.h.p., $|G_{\ell}(u)| \leq A \log n$ for all $u \in V(G)$.
\end{lemma}
 
\noindent The proof of this lemma is given at the end of this section.
 Hence, we can assume $G_{\ell}(u) \leq A \log n$ for all $u \in V(G)$ in the rest of the proof.

We now call a vertex $u$ {\it nice} if $G_{\ell}(u)$ is a tree and the height of $G_{\ell}(u)$, denoted by $D_{\ell}(u)$, is at least $\widehat{\alpha}_{\epsilon} \log n$, i.e., $D_{\ell}(u) \geq \widehat{\alpha}_{\epsilon} \log n$.

\noindent Note that when $u$ is nice,  then there exists a node $v$ at level $\widehat{\alpha}_{\epsilon} \log n$ behind $u$ such that  $w(\gamma_u(v)) \leq  \frac{1-\epsilon}{d-2} \log n$, where $\gamma_u (v)$ denote the path from $u$ to $v$ in $B(u, .)$. Using the second moment method, we now prove that there exists at least one nice vertex.

Let $N_a$ denote the event that node $a$ is nice, and $X = \sum_{a\in [n]} \ind( N_a)$ be the total number of nice vertices. We now show that $X \geq 1$ w.h.p., which concludes the proof.

\noindent Let $Z$ be the distribution of the number of $a$-good nodes at level $\ell$ in $(d-1)$-array tree having $a$ as a root.  Conditioning on the tree structure of $G_{\ell}(a)$ and by removing one of the children of $a$ (and all its descendants), the set of $a$-good nodes are distributed as a branching process with distribution $Z$. Note that by (\ref{ineq-proba-two}), we have
 \begin{eqnarray}
 \EE Z = (d-1)^{\ell} \PP\left(w(\gamma_{\ell}) \leq \frac{\ell(1-\epsilon)}{(d-2)\widehat{\alpha}_{\epsilon}} \right) \geq 
 \exp\left(\left(-1+\frac{\epsilon}{d-2}\right) \frac{\ell}{\widehat{\alpha}_{\epsilon}}\right) .
 \end{eqnarray}

\noindent Let $P_k(a)$ be the probability that this branching process survives for at least $k$ generations. By basic recurrent argument, we have
 $$P_{k+1}(a) = 1 - \Phi_Z(1-P_k(a)),$$
where $\Phi_Z(s) = \EE s^Z$ denote the generation function of $Z$.

\noindent Note that $\EE Z <1$ (for $\epsilon$ small enough) and the branching process is subcritical. Hence, $P_k(a) \to 0$ as $k \to \infty$. Using $1 - \Phi_Z(1-x) =  \Phi_Z'(1) x + O(x^2)$, and $ \Phi_Z'(1) = \EE Z$, it follows easily that
\begin{eqnarray}
P_k (a) = (\EE Z + o(1))^k, \ \mbox{as} \ k \to \infty.
\end{eqnarray}

\noindent Thus, conditioning on the tree-structure of $G_{\ell} (a)$ (and by choosing $k= \frac{\widehat{\alpha} _{\epsilon}}{\ell} \log n$), we get
$$\PP (D_{\ell}(a) \geq \widehat{\alpha} _{\epsilon} \log n) \geq (1 \pm o(1)) n^{-1+\frac{\epsilon}{d-2}}.$$

\noindent Since the size of $G_{\ell}(a)$ is (w.h.p.)\ smaller that $A\log n$ (by Lemma~\ref{lem-boundLog}),  with probability at least $1- O(\log n / n)$, $G_{\ell}(a)$ is a tree.

\noindent Putting all these together, we have
$$\EE X = \sum_a \PP(N_a) \geq  \frac{2}{3} n^{\frac{\epsilon}{d-2}}.$$

\noindent And,
\begin{eqnarray*}
\EE X^2 &=& \EE (\sum_a \ind(N_a))^2  =  \EE \sum_{a,b} \ind(N_a) \ind(N_b) \\
&=& \EE \left[\sum_{a}\ind(N_a) \sum_{b:\ G_{\ell}(a)\cap G_{\ell}(b)\neq \emptyset}  \ind(N_b) + \sum_{a,b:\ G_{\ell}(a)\cap G_{\ell}(b) = \emptyset} \ind(N_a) \ind(N_b)\right]\\
&\leq& (A\log n)^2 \EE X +  (\EE X)^2,
\end{eqnarray*}
\noindent where the last inequality follows by Lemma~\ref{lem-boundLog}.
We conclude that
$$Var[X] = \EE X^2 - (\EE X)^2 \leq (A \log n)^2 \EE X .$$

\noindent Then, by Chebysev's inequality w.h.p.\ $X \geq \frac{1}{2} n^{\frac{\epsilon}{d-2}}$.

\noindent This completes the proof of the lower bound. 

\noindent We end this section by presenting the proof of Lemma~\ref{lem-lower} and  Lemma~\ref{lem-boundLog}.

\begin{proof}[Proof of Lemma~\ref{lem-lower}]
We condition on the path $\gamma_k$ between  $v_0$ and $v_k$. We now remove the path $\gamma_k$ and consider the exploration process defined is Section~\ref{sse:explor} starting from $v_0$. (The proof is similar to \cite[Lemma 3.5]{DKLP09}.)
\\
Let $\tau_i$ denote the time of the $i$'th exploration step (for $i \geq 0$, $\tau_0 = 0$).  Note that $\tau_{i+1} - \tau_i \geq_{st} Y_i$, where $Y_i$ are independent exponential random variables with $$\EE [Y_i] =  \left(1 + (d-2)(i+1)\right)^{-1}.$$
Note that this is true since the worst case is when $X_i(a) = 0$, i.e., the explored set forms a tree.

We let $z=\lfloor \sqrt{n/\log n}\rfloor$. We will show later that the growing balls in the exploration process starting from $v_0$ and $v_k$ will not intersect w.h.p.\ provided that they are of size less than $z$. We now prove that $\tau_z > \frac{1-\epsilon}{2(d-2)} \log n$ with high probability.

We have
\begin{eqnarray*}
\PP(\tau_z \leq t) &\leq& \int_{\sum_{i=1}^z x_i \leq t} \prod_{i=1}^z [1+(d-2)i]e^{-\sum_{i=1}^z (1+(d-2)i)x_i} dx_1 \dots dx_z \\
&=& \int_{0 \leq y_1 \leq \dots \leq y_z \leq t} \prod_{i=1}^z [1+(d-2)i] e^{-y_z} e^{-(d-2)\sum_{i=1}^z y_i} dy_1 \dots dy_z,
\end{eqnarray*}
where $y_k = \sum_{i=0}^{k-1}x_{z-i}$. Letting $y=y_z$ and accounting for all permutations over $y_1, \dots, y_{z-1}$ (by giving to these variables the range $[0,y]$), we obtain
\begin{eqnarray*}
\PP(\tau_z \leq t) &\leq& \int_0^t e^{-(d-1)y} \frac{\prod_{i=1}^z (i+\frac{1}{d-2})}{(z-1)!} \\
&&  \ \ \ . \left(\int_{[0,y]^{z-1}} (d-2)^z e^{-(d-2)\sum_{i=1}^{z-1}y_i} dy_1 \dots dy_{z-1}\right) dy \\
&\leq& \int_0^t e^{-(d-1)y} \frac{\prod_{i=1}^z (i+\frac{1}{d-2})}{(z-1)!} . \left( \prod_{i=1}^{z-1} \int_o^y (d-2) e^{-(d-2)y_i} dy_i\right) dy\\
&\leq& C (d-2) z^{\frac{d-1}{d-2}} \int_0^t  e^{-(d-1)y} \left(1-e^{-(d-2)y}\right)^{z-1}dy,
\end{eqnarray*}
where $C>0$ is an absolute constant. Now using the fact that $\left(1-e^{-(d-2)y}\right)^{z-1} \leq e^{-n^{\alpha}}$, for some $\alpha >0$ and for all $0 \leq y \leq \frac{1-\epsilon}{2(d-2)} \log n =: t_0$, we obtain
$$\PP \left( \tau_z  \leq  \frac{1-\epsilon}{2(d-2)} \log n\right) \leq C (d-2) z^{\frac{d-1}{d-2}} \int_0^{t_0} e^{-n^{\alpha}} dy = o(n^{-4}).$$

Similarly considering the exploration process for $v_k$, again after time $t_0$, we obtain w.h.p.\ a set of size at most $z$.
Now remark that, because each matching is uniform among the remaining half-edges, the probability of hitting the ball of size $t_0$ around $v_0$ is at most $z/n$. Altogether,
$$\PP(\gamma_k \neq \pi(v_0,v_k) \mid w(\gamma_k)\leq \frac{1-\epsilon}{d-2}\log n ) \leq \frac{z^2}{n} + o(1) = o(1),$$
as desired.
\end{proof}

\begin{proof}[Proof of Lemma~\ref{lem-boundLog}]
Let $Z_{\ell}(a)$ denote the number of $a$-good nodes in $B_{\ell}(a)$ (i.e., the nodes in generation $\ell$ behind $a$ with (weighted) distance smaller than $\frac{\ell(1-\epsilon)}{(d-2)\widehat{\alpha} _{\epsilon}}$ from $a$). By Markov inequality and from (\ref{ineq-proba-two}), we obtain
\begin{eqnarray*}
\PP(Z_{\ell}(a) \geq 1) &\leq& \EE Z_{\ell}(a) \\
&\leq& d(d-1)^{\ell-1} (d-1)^{-\ell} \exp\left(\left(-1+\frac{\epsilon}{d-2}\right) \frac{\ell}{\widehat{\alpha}_{\epsilon}}\right) \\
&=& \frac{d}{d-1} \exp\left(\left(-1+\frac{\epsilon}{d-2}\right) \frac{\ell}{\widehat{\alpha}_{\epsilon}}\right) =: \beta_{\epsilon}
\end{eqnarray*}
(this follows from the fact that the worst case is when $B_{\ell}(a)$ forms a tree).

\noindent Thus, for $\ell=\ell_{\epsilon}$ large enough, we have $\PP(Z_{\ell}(a) \geq 1) \leq \beta_{\epsilon} < 1$. 

\noindent We conclude (for any integer $K$)
\begin{eqnarray*}
\PP\left(G_{\ell}(a)\leq K d(d-1)^{\ell-1}\right) \leq \beta_{\epsilon}^K.
\end{eqnarray*}

\noindent Now by choosing $K= 2\log n/|\log \beta_{\epsilon}|$, we get
\begin{eqnarray*}
\PP\left(G_{\ell}(a)\leq 2 d(d-1)^{\ell-1}\log n/|\log \beta_{\epsilon}|\right) \leq n^{-2}.
\end{eqnarray*}

\noindent Taking a union bound over all $a$ finishes the proof.

\end{proof}

\section*{Acknowledgements}
Part of this work was done when the authors were visiting MSRI, Berkeley. We thank them for their hospitality. We also thank Shankar Bhamidi for helpful comments.
Hamed Amini gratefully acknowledges financial support from the Austrian Science Fund (FWF) though project P21709.

\bibliographystyle{abbrv}

\bibliography{Biblio}
\end{document}